\documentclass{amsart}

\usepackage{amsmath}
\usepackage{amssymb}
\usepackage{amsthm}
\usepackage{amscd}

\usepackage[latin1]{inputenc}
\usepackage{graphicx}

\renewcommand{\subjclassname}{AMS \textup{2010} Mathematics Subject
Classification\ }

\newtheorem{teor}{Theorem}
\newtheorem{cor}{Corollary}
\newtheorem{prop}{Proposition}
\newtheorem{lem}{Lemma}
\theoremstyle{definition}
\newtheorem*{exa}{Example}
\newtheorem*{rem}{Remark}

\author{Antonio M. Oller-Marc\'{e}n}
\title{On arithmetic numbers}
\address{Departamento de Matem\'{a}ticas, Universidad de Zaragoza\\
C/Pedro Cerbuna 12, 50009 Zaragoza (Espa\~{n}a)} \email{oller@unizar.es}

\begin{document}
\maketitle

\begin{abstract}
An integer $n$ is said to be \textit{arithmetic} if the arithmetic mean of its divisors is an integer. In this paper, using properties of the factorization of values of cyclotomic polynomials, we characterize arithmetic numbers. As an application, in Section 2, we give an interesting characterization of Mersenne numbers.
\end{abstract}
\subjclassname{11A25, 11B99}\\


\section{Introduction}
For an integer $n$ we can define \cite{ORE} the arithmetic function $A(n)$ as the arithmetic mean of the divisors of $n$; i.e., $A(n)=\frac{\sigma(n)}{\tau(n)}$. An integer $n$ is then said to be \textit{arithmetic} \cite[B2]{GUY} if $A(n)$ is an integer (see sequence A003601 in OEIS).

Ore \cite{ORE} characterized square-free arithmetic numbers. The set of arithmetic numbers has density 1 \cite{POM} and Bateman et al. \cite{BEP} have studied the distribution of non-arithmetic numbers. Nevertheless, we have not been able to find in the literature a general solution to the problem of the characterization of arithmetic numbers. 

Since $A(n)$ is an arithmetic function, it is natural to study the case when $n=p^k$ is a prime power. In this case we can easily give an explicit expression for $A(n)$. Namely:
$$A(p^{k})=\frac{1}{k+1}\sum_{i=0}^{k} p^i=\frac{p^{k+1}-1}{(k+1)(p-1)}=\frac{1}{k+1}\prod_{1\neq d|k+1} \Phi_{d}(p),$$
where $\Phi_d$ denotes, as usual, the $d$-th cyclotomic polynomial.

From the above expression it is quite clear that the prime factorization of numbers of the form $\Phi_d(a)$ will play a key role. In particular, the following classical result \cite{ROI} will be useful.

\begin{teor}
Let $a,n\geq 2$ be integers and let $p$ be the largest prime factor of $n$. Put $n=p^km$, then:
\begin{itemize}
\item[i)] $p$ is a prime factor of $\Phi_n(a)$ if and only if $\textrm{ord}_p(a)=m$ (hence $m$ divides $p-1$). Moreover, in this case, $p^2$ does not divide $\Phi_n(a)$.
\item[ii)] If $q$ is another prime dividing $\Phi_n(a)$, then $\textrm{ord}_q(a)=n$. Moreover, in this case, $q$ does not divide $n$ if and only if $q\equiv 1$ (mod $n$).
\end{itemize}
\end{teor} 

\section{Arithmetic prime powers}
The main goal of this section is to find out when the prime-power $p^{k}$ is arithmetic. We will start considering the case when $k+1$ is also a prime power. We have the following result.

\begin{prop}
Let $p$ be a prime and let $k$ be an integer such that $k+1=q^m$ is a prime power. Then $A(p^k)\in\mathbb{Z}$ if and only if $q$ divide $p-1$.
\end{prop}
\begin{proof}
First observe that:
$$A(p^k)=\frac{1}{k+1}\frac{p^{k+1}-1}{p-1}=\frac{1}{q^m}\prod_{1\neq d|q^m}\Phi_d(p)=\frac{1}{q^m}\prod_{j=1}^m\Phi_{q^j}(p).$$
By Theorem 1 i), if $q$ divides $p-1$; i.e., if $\textrm{ord}_q(p)=1$, then $q$ divides $\Phi_{q^j}(p)$ for every $1\leq j\leq m$ and hence $A(p^k)\in\mathbb{Z}$.

Conversely, if $A(p^k)\in\mathbb{Z}$, it follows that $p^{q^m}-1\equiv 0$ (mod $q$). This clearly implies ($q^m$ and $q-1$ being coprime) that $p-1\equiv 0$ (mod $q$) and the result follows.
\end{proof}

Let us introduce some notation. Given an integer $n$ and its prime power decomposition $n=q_1^{m_1}\cdots q_r^{m_r}$, we define $d_j(n):=\gcd(q_j-1,n)$. 

\begin{rem}
If $n=q_1^{m_1}\cdots q_r^{m_r}$ we can assume that $q_1<\cdots<q_r$. If we denote by $n_j=q_1^{m_1}\cdots q_{j-1}^{m_{j-1}}$ ($n_1=1$), it is easy to see that
$$\gcd(q_j-1,n)=\gcd(q_j-1,n_j)=\gcd(q_j-1, n/q_j^{m_j}),$$
because $q_k$ cannot divide $q_j-1$ for any $k\geq j$.
\end{rem}

We can now prove the following result.

\begin{prop}
Let $p$ be a prime and $k$ be any integer. If $k+1=q_1^{m_1}\cdots q_r^{m_r}$ is the prime power decomposition of $k+1$, we have that $A(p^k)\in\mathbb{Z}$ if and only if $q_j|p^{d_j(k+1)}-1$ for every $j=1,\dots,r$.
\end{prop}
\begin{proof}
In this case $A(p^k)=\displaystyle{\frac{1}{k+1}\prod_{1\neq d|k+1}\Phi_d(p)}$. If $A(p^k)\in\mathbb{Z}$ it follows that $q_j$ divides $p^{k+1}-1$ for every $j$. This imples that $q_j$ also divides $p^{\gcd(q_j-1,k+1)}-1$ as claimed.

Conversely, assume that $q_j|p^{d_j(k+1)}-1$ for every $j=1,\dots,r$. This implies that $\textrm{ord}_{q_j}(p)$ divides $d_j(k+1)$. Now, if we put $D_{(j,i)}=\textrm{ord}_{q_j}(p)q_j^{i}$ we have that $D_{(j,i)}$ is a divisor of $k+1$ and $q_j$ is its largest prime factor (see the previous remark). We can thus apply Theorem 1 i) to conclude that $q_j$ divides $\Phi_{D_{(j,i)}}(p)$ for every $1\leq j\leq r$ and for every $1\leq i\leq m_j$. Hence $q_j^{m_j}$ divides $\displaystyle{\prod_{1\neq d|n+1}\Phi_d(p)}$ for every $j$ and the proof is complete.
\end{proof}

Recall that he \textit{radical} of an integer is defined to be its largest square-free divisor. Namely, if $n=q_1^{m_1}\cdots q_r^{m_r}$ then $\textrm{rad}(n)=q_1\dots q_r$. Now, let us define $\Delta (n)=\textrm{lcm}(d_1(n),\dots, d_r(n))$. Recalling the definition of $d_j(n)$ it is clear that $\Delta(n)=\gcd(n,\textrm{lcm}(q_j-1))=\gcd(n,\lambda(\textrm{rad}(n)))$, where $\lambda$ is Carmichael's function (see A173751 in OEIS). 

Observe that, from the definition of $d_j(n)$, we have that $\Delta(n)=q_1^{\mu_1}\cdots q_{r-1}^{\mu_{r-1}}$ with $0\leq \mu_j\leq m_j$ for all $j=1,\dots, r-1$. This observation allows us to prove the following lemma.

\begin{lem}
For every $j=1,\dots, r$, we have that $d_j(n)=\gcd(q_j-1,\Delta(n)/q_j^{\mu_j})$
\end{lem}
\begin{proof}
By definition $d_j(n)$ divides $\Delta(n)$ and since $q_j$ cannot appear in its prime power decomposition, it clearly divides $\Delta(n)/q_j^{\mu_j}$. On the other hand $\gcd(q_j-1,\Delta(n)/q_j^{\mu_j})$ must be of the form $q_1^{e_1}\dots q_{j-1}^{e_{j-1}}$ with $0\leq e_j\leq \mu_j\leq m_j$. Thus, it divides $n_j$ and the result follows.
\end{proof}

From the previous corollary it follows readily that if a prime power $p^{k}$ is arithmetic, then $\textrm{rad}(k+1)$ divides $p^{\Delta(k+1)}-1$. The main result of this section is the following theorem which proves that the converse is also true.

\begin{teor}
Let $p$ be a prime and $k$ be an integer. Then, $A(p^k)\in\mathbb{Z}$ if and only if $\textrm{rad}(k+1)$ divides $p^{\Delta(k+1)}-1$.
\end{teor}
\begin{proof}
If $\textrm{rad}(k+1)$ divides $p^{\Delta(k+1)}-1$, then $q_j$ divides $p^{\Delta(k+1)}-1$ for every $j$. Since $p^{\Delta(k+1)}\equiv p^{\Delta(k+1)/q_j^{\mu_j}} \equiv 1\equiv p^{q_j-1}$ (mod $q_j$) it follows that $q_j$ divides $p^{\gcd(q_j-1,\Delta(k+1)/q_j^{\mu_j})}-1$ so it is enough to apply the previous lemma together with Proposition 2.
\end{proof}

The rest of the section will be devoted to present some applications of the previous results.

In \cite{THR}, the arithmetic mean of the core divisors of a number $A^{*}(n)$ is considered, where a \textit{core divisor} is one which is a multiple of $\textrm{rad}(n)$. Among other results it is proved that $A^{*}(p^{p})$ is integral for any prime $p$. Let us see that if $p\neq 2$ this result also holds when considering all the divisors.

\begin{prop}
If $p$ is an odd prime, then $p^{p}$ is arithmetic.
\end{prop}
\begin{proof}
We can write $p+1=q_1^{m_1}\cdots q_r^{m_r}$ with $q_1=2$. Since $q_j-1$ is even for every $2\leq j\leq r$, it follows that $\Delta(p+1)$ is also even. Observe that $p\equiv -1\ (\textrm{mod}\ q_j)$ and thus, $p^{\Delta(p+1)}\equiv (-1)^{\Delta(p+1)}\equiv 1\ (\textrm{mod}\ q_j)$. 

The previous reasoning does not work if $r=1$, but in such case $p+1$ is a power of 2 and it is enough to apply Proposition 1 since $p-1$ is even.
\end{proof}

Of course, if $p,q$ are odd primes, $p^q$ is not arithmetic in general; e.g., $A(3^5)=\frac{182}{3}$. Nevertheless we have the following proposition which was already suggested by the proof of the previous one.

\begin{prop}
If $p$ is an odd prime and $m$ is a Mersenne number, then $p^m$ is arithmetic. In particular $p^q$ is arithmetic for every Mersenne prime $q$.
\end{prop}
\begin{proof}
In this case $m+1$ is a power of 2 and $p-1$ is even, so we can apply Proposition 1. 
\end{proof}

Before we pass to the following section we will see that, in fact, the previous proposition gives us an interesting characterization of Mersenne numbers.

\begin{cor}
Let $m$ be any integer. Then $p^m$ is arithmetic for every odd prime $p$ if and only if $m$ is a Mersenne number.
\end{cor}
\begin{proof}
If $p^m$ is arithmetic for every odd prime, then $m+1$ divides $p^{m+1}-1$ which implies that $\gcd(m+1,p)=1$. Thus $m+1$ must be a power of 2 as desired. The converse is given by the previous proposition.
\end{proof}

\section{The general case}

To give general conditions for any integer $n$ to be arithmetic is a more difficult task. Since $A(n)$ is an arithmetic function we can use the results given in the previous section to obtain the following strightforward result.

\begin{cor}
Let $p_1,\dots, p_r$ be odd prime numbers and let $n_1,\dots, n_r$ be integers such that $\textrm{rad}(n_j+1)$ divides $p_j^{\Delta(n_j+1)}-1$ for every $1\leq j\leq r$. If $N=p_1^{n_1}\cdots p_r^{n_r}$, then $A(N)\in\mathbb{Z}$. In particular $p_1^{p_1}\cdots p_r^{p^r}$ is arithmetic and, if $m_1,\dots, m_r$ are Mersenne numbers then $p_1^{m_1}\cdots p_r^{m_r}$ is also arithmetic.
\end{cor}

In \cite{ORE} the square-free case was completely solved since it easily follows from the definition of $A(n)$ that an odd square-free number is always arithmetic and an even square-free number is arithmetic if and only if one of its prime divisors is of the form $4k-1$. In \cite{THR} it was proved that $A^{*}(n)$ is integral if $n$ is cube-free. Of course this fact does not remain true when considering all the divisors of $n$; e.g., $A(75)=\frac{62}{3}$. We will start the section characterizing cube-free arithmetic numbers. To do so we first need to prove the following technical lemma.

\begin{lem}
Let $p$ be a prime. If 3 divides $1+p+p^2$, then $9$ does not divide $1+p+p^2$.
\end{lem}
\begin{proof}
Recall that 3 divides $1+p+p^2$ if and only if $p=3k+1$, but it that case $1+p+p^2=9k^2+9k+3$ is not a multiple of 9.
\end{proof}

\begin{prop}
Let $n=2^a3^bp_1\cdots p_rq_1^2\cdots q_s^2$ be a cube-free integer. Let $\alpha=\textrm{card}\{q_i\ |\ q_i\equiv 2\ (mod\ 3)\}$. Then:
\begin{itemize}
\item If $a\neq 1$, then $A(n)\in\mathbb{Z}$ if and only if $3^{\alpha+\left[\frac{a}{2}\right]+\left[\frac{b}{2}\right]}$ divides $\prod(p_i+1)$.
\item If $a=1$ and $b\neq 1$, then $A(n)\in\mathbb{Z}$ if and only if $3^{\alpha-1+\left[\frac{b}{2}\right]}$ divides $\prod(p_i+1)$ and there exists $j\in\{1,\dots,r\}$ such that $p_j=4k-1$.
\item If $a=b=1$, then $A(n)\in\mathbb{Z}$ if and only if $3^{\alpha-1}$ divides $\prod(p_i+1)$.
\end{itemize}
\end{prop}
\begin{proof}
Observe that $\displaystyle{A(n)=\frac{\sum_{i=0}^{a}2^i}{a+1}\frac{\sum_{i=0}^{b}3^i}{b+1}\prod_{i=1}^r\frac{p_i+1}{2}\prod_{i=1}^s\frac{1+q_i+q_i^2}{3}}$, with $0\leq a,b\leq 2$ and where the third factor is always an integer. Then it is enough to apply Proposition 1 and the previous lemma; also noting that $1+q_i+q_i^2$ is always odd.
\end{proof}

Before we proceed let us introduce some notation. If $N=p_1^{n_1}\cdots p_r^{n_r}$, let $\{q_1<\cdots<q_s\}$ be the set of primes appearing in the factorizations of $n_1+1,\dots,n_r+1$. Thus, for every $i\in\{1,\dots,r\}$ we can put $n_i+1=q_1^{a_{i,1}}\cdots q_s^{a_{i,s}}$ with $0\leq a_{i,j}$. Also, for every $i\in\{1,\dots,r\}$ and $j\in\{1,\dots,s\}$, let us define $\alpha_{i,j}=\textrm{ord}_{q_i}(p_j)$. Observe that $\alpha_{i,j}$ cannot contain any prime larger than $q_i$ because $\alpha_{i,j}|q_i-1$. We also introduce the following sets for every $i\in\{1,\dots,s\}$:
$$J(i):=\{j : \alpha_{i,j}|n_j+1\},$$
$$E(i):=\{j : q_i|n_j+1\}.$$
Finally, for every integer $n$ and prime $p$, $|n|_p$ denotes the exponent of $p$ in the prime power decomoposition of $n$.

With this notation we have the following result.

\begin{teor}
$A(N)\in\mathbb{Z}$ if and only if the following conditions hold for every $i$:
\begin{itemize}
\item[a)] $J(i)\neq\emptyset$,
\item[b)] $\displaystyle{\sum_{j=1}^{r} a_{j,i}\leq \sum_{j\in J(i)\cap E(i)} a_{j,i} + \sum_{j\in J(i)\setminus E(i)} \left|\prod_{1\neq d|n_j+1}\Phi_d(p_j)\right|_{q_i}.}$
\end{itemize}
\end{teor}
\begin{proof}
First of all observe that
$$A(N)=\frac{\displaystyle{\prod_{k=1}^r\left(\prod_{1\neq d|n_k+1}\Phi_d(p_k)\right)}}{\displaystyle{\prod_{k=1}^{s}\left(q_k^{\sum_{j=1}^{r}a_{j,k}}\right)}}.$$
Now, assume that $A(N)\in\mathbb{Z}$ and fix $q_i$ for some $1\leq i\leq s$. It follows that $q_i$ divides $\Phi_d(p_j)$ with $d|n_j+1$ for some $1\leq j\leq r$ and three cases arise:
\begin{itemize}
\item[i)] $q_i\not| d$. In this cases Theorem 1 ii) applies to obtain that $\alpha_{i,j}=d$ divides $n_j+1$.
\item[ii)] $q_i|d$ and it is the largest prime factor of $d$. If $d=q_i^{\epsilon}d'$ Theorem 1 i) implies that $\alpha_{i,j}=d'$ divides $n_j+1$.
\item[iii)] $q_i|d$ and $d$ contains a prime factor $q_k$ larger or equal that $q_i$. Theorem 1 i) implies that $q_k|d|q_i-1$ which is a contradiction.
\end{itemize}
We have thus seen that $\alpha_{i,j}$ divides $n_j+1$ for some $j$; i.e., that $J(i)\neq\emptyset$ and a) is proved.

If $j\not\in J(i)$, then $\alpha_{i,j}$ does not divide $n_j+1$ and Theorem 1 implies that $q_i$ cannot divide $\Phi_d(p_j)$ for any divisor $d$ of $n_j+1$.
Now, if $j\in J(i)\cap E(i)$, Theorem 1 i) implies that $q_i^{a_{j,i}}$ is the largest power of $q_i$ dividing $\prod_{1\neq d|n_j+1}\Phi_d(p_j)$.
Finally, if $j\in J(i)\setminus E(i)$ it follows that $q_i$ divides $\Phi_{\alpha_{i,j}}(p_j)$. This proves b).

The converse also follows from Theorem 1 and we give no further details.
\end{proof}

If, in the previous result we assume $n_1+1,\dots,n_r+1$ to be distinct primes, we obtain the following proposition. Although it is a consequence of Theorem 3, we will give a self-contained proof. 

\begin{prop}
Let $p_1,\dots,p_r$ be distinct primes and let $q_1<\cdots< q_r$ also be primes. Put $n=p_1^{q_1-1}\cdots p_r^{q_r-1}$. Then $A(n)\in\mathbb{Z}$ if and only if for every $i\in\{1,\dots,r\}$ either $q_i|p_i-1$ or there exists $j<i$ such that $\textrm{ord}_{q_i}(p_j)=q_j$ (hence $q_j|q_i-1$).
\end{prop}
\begin{proof}
Observe that $A(n)=\displaystyle{\frac{\Phi_{q_1}(p_1)\cdots\Phi_{q_r}(p_r)}{q_1\cdots q_r}}$. Thus $A(n)\in\mathbb{Z}$ if and only if $q_i$ divides $\displaystyle{\prod_{j=1}^r\Phi_{q_j}(p_j)}$ for every $1\leq i\leq r$. Now, fix $i$ and assume that $q_i$ divides $\Phi_{q_j}(p_j)$ for some $1\leq j\leq r$. Then, two cases arise:
\begin{itemize}
\item[i)] $i=j$. Due to Proposition 1, this happens if and only if $q_i$ divides $p_i-1$.
\item[ii)] $i\neq j$. Theorem 1 ii) implies that $\textrm{ord}_{q_i}(p_j)=q_j$ and $q_i\equiv 1$ (mod $q_j$) (and consequently $j<i$).
\end{itemize}
The converse is obvious since $\textrm{ord}_{q_i}(p_j)=q_j$ clearly implies that $q_i$ divides $\Phi_{q_j}(p_j)$ and the proof is complete.
\end{proof}

We will close the paper with a necessary condition for an integer to be arithmetic. It is a consequence of Theorem 3, so we will keep using the same notation.

\begin{cor}
Let $N=p_1^{n_1}\cdots p_r^{n_r}$ with $p_1,\dots,p_r$ being distinct primes and $n_1,\dots, n_r$ being any integers. Let us denote by $\mathcal{Q}$ the set of primes appearing in the factorizations of $n_1+1,\dots, n_r+1$ and put $q_1=\min\mathcal{Q}$. Assume that $q|n_k+1$ for a unique $k$. In this situation if $n$ is arithmetic, then $q$ divides $p_k^{n_k+1}-1$. 
\end{cor}
\begin{proof}
With the notation of Theorem 3, we have that $E(1)=\{k\}$; i.e., $a_{i,1}=0$ for every $i\neq k$. Thus $J(1)\cap E(1)=\{k\}$ if $\alpha_{1,k}$ divides $n_k+1$ and empty otherwise. 

Assume that $q_1$ does not divide $p_k^{n_k+1}-1$. This means that $\alpha_{1,k}\not\in J(1)$ so, since $J(1)\neq\emptyset$ there must exist $h\neq k\in J(1)$. Consequently $\alpha_{1,h}$ divides $n_h+1$ but, since $\gcd(q_1-1,n_h+1)=1$ (recall that $q_1=\min\mathcal{Q}$) it follows that $p_h\equiv 1$ (mod $q_1$). This clearly implies that $q_1$ divides $n_h+1$; i.e., that $h\in E(1)=\{k\}$. A contradiction.
\end{proof}

\begin{rem}
Observe that we can always apply the previous corollary if $\gcd(n_i+1,n_j+1)=1$ for all $i,j$, but if $q_1=2$ it is only useful when $p_k=2$.
\end{rem}

\begin{exa}
Let $n=3^{34}5^87^{24}$. In this case $\min\mathcal{P}=3$ and it only divides $n_2+1=9$ and we can apply the previous proposition. Since 3 does not divide $5^9-1$ we conclude that $n$ is not arithmetic.
\end{exa}

\bibliography{./refsumpot}
\bibliographystyle{plain}

\end{document}